\documentclass{article}%
\usepackage{amsmath}
\usepackage{amssymb}
\usepackage{euler}
\usepackage{amsfonts}
\usepackage{graphicx}%
\providecommand{\U}[1]{\protect\rule{.1in}{.1in}}
\newtheorem{theorem}{Theorem}

\newtheorem{corollary}[theorem]{Corollary}

\newtheorem{lemma}[theorem]{Lemma}

\newtheorem{proposition}[theorem]{Proposition}

\newenvironment{proof}[1][Proof]{\noindent\textbf{#1.} }{\ \rule{0.5em}{0.5em}}
\begin{document}

\title{\textsf{Constructive Notes on Locally Convex Spaces}}
\author{\textsf{Douglas S. Bridges}}
\maketitle

\begin{abstract}%
\noindent
\textsf{We give a detailed, corrected presentation of some fundamentals of the
constructive theory of locally convex spaces that appear without proofs
in\ \cite[Section 5.4]{BVtech}. This suffices for some important functional
analytic theorems that are stated in our final section.}

\end{abstract}

%

\normalfont\sf
%

\bigskip
%

\bigskip
%

\bigskip

\section{Introduction}

At the foot of pages 129 in \cite{BVtech}, in our discussion of locally convex
spaces in constructive analysis, we wrote:

\begin{quote}
The proofs of the next five results are similar to those of their counterparts
in metric space theory ... and are left as an exercise.
\end{quote}

%

\noindent
It turns out that some of those proofs are not quite so straightforward as our
quote suggests, and one theorem appears not to be constructively
provable.\footnote{%
\normalfont\sf
Theorem 5.4.6 of \cite{BVtech} seems unlikely to hold constructively as
stated. We prove it under the restriction that the defining family of
seminorms is countable (Corollary \ref{c22}), which classically is equivalent
to the space being metrisable. In the general case, we can replace
\emph{totally bounded} in the conclusion by $F$\emph{-totally bounded}, where
$F$ is a given finitely enumerable subset of the index set of the defining
family of seminorms (Theorem \ref{t11}). This version of the theorem seems
sufficiently powerful in the constructive context.} In the present
article\footnote{%
\normalfont\sf
\textbf{Keywords: \ }locally convex, constructive \ \ \textbf{MSC
classification: \ }46S30} we provide proofs of those results (corrected as
necessary), as well as correcting some others in \cite[Section 5.4]{BVtech}.
To do so, we give an amplified presentation of some fundamental elements of
the theory of locally convex spaces within the framework of Bishop's
constructive analysis (a good introduction to which is given by the articles
\cite{Haj1,Haj2} in \cite{Handbook}).

\section{Locally convex topologies}

A \textbf{locally convex space} is a pair $\left(  X,\left(  p_{i}\right)
_{i\in I}\right)  $, where

\begin{itemize}
\item $X$ is a real or complex linear space, with inequality relation $\neq$;

\item for each $i$ in the index set $I$, $p_{i}$ is a seminorm on $X$;

\item $x\neq0$ if and only if $p_{i}(x)>0$ for some $i\in I$.\footnote{%
\normalfont\sf
We omitted this requirement in \cite[Section 5.4]{BVtech}.}
\end{itemize}

%

\noindent
We call the functions $p_{i}$ the \textbf{defining seminorms} of the locally
convex space. If it is clear what the defining seminorms are, we refer to $X$
itself as a locally convex space. Given an inhabited\footnote{%
\normalfont\sf
From now on we shall assume that\emph{\ finitely enumerable} means
\emph{inhabited and finitely enumerable.}} finitely enumerable subset of $I$,
a point $x_{0}\in X$, and $r>0$, we define the \textbf{open and closed }%
$F$\textbf{-balls} with centre $x_{0}$ and radius $r$ to be, respectively,%
\begin{align*}
B^{F}(x_{0},r)  &  \equiv\left\{  x\in X:\sum_{i\in F}p_{i}(x-x_{0}%
)<r\right\}  \text{ and}\\
\overline{B}^{F}(x_{0},r)  &  \equiv\left\{  x\in X:\sum_{i\in F}p_{i}%
(x-x_{0})\leq r\right\}  .
\end{align*}
The \textbf{locally convex topology} $\tau_{X}$ on $X$ consists of all finite
unions of open $F$-balls. We also denote $B^{F}(0,r)$ by $B^{F}(r)$, and
$\overline{B}^{F}(0,r)$ by $\overline{B}^{F}(r)$.

Let $S\subset X$. The \textbf{closure }of $S$ in $X$ is the set $S^{c}$ of all
$x\in X$ such that $S\cap B^{F}(x,r)$ is inhabited for each finitely
enumerable $F\subset I$ and each $\varepsilon>0$. We say that $S$ is
\textbf{closed }in $X$ if $S=S^{c}$; \textbf{dense }in $X$ if $S^{c}=X$; and
\textbf{separable} if it has a countable dense subset. On the other hand, $S$
is \textbf{bounded} if there exist $c>0$, a finitely enumerable $F\subset I$,
and $r>0$ such that $S\subset cB^{F}(r)$.

For example, a normed linear space $\left(  X,\left\Vert \ \right\Vert
\right)  $ is a locally convex space in which the family of defining seminorms
comprises the single norm $\left\Vert \ \right\Vert $, and the locally convex
topology is just the standard metric topology associated with that
norm.\footnote{%
\normalfont\sf
From now on, when we write $\mathbb{R}$ or $\mathbb{C}$, we are thinking of
these spaces as locally convex relative to the single defining seminorm
$x\rightsquigarrow\left\vert \ x\right\vert $.} If $Y$ is also a normed linear
space, then the linear space $\mathcal{B}(X,Y)$ of bounded linear mappings of
$X$ into $Y$ has a locally convex structure with defining seminorms
$\left\Vert \;\right\Vert _{x}:T\rightsquigarrow\left\Vert Tx\right\Vert
$,$\ $indexed by the vectors $x\in X$ with $\left\Vert x\right\Vert \leq1$. In
the special case where $Y$ is the groundfield ($\mathbb{R}$ or $\mathbb{C}$)
the locally convex structure gives us the \textbf{weak}$^{\ast}$%
\textbf{\ topology} on the dual space $X^{\ast}$ of all bounded linear
functionals on $X$. If $X=Y=H$, where $H$ is a Hilbert space, then the
seminorms $\left\Vert \ \right\Vert _{x}$ with $x\in H$ and $\left\Vert
x\right\Vert \leq1$ give us the \textbf{strong operator topology} on the space
$\mathcal{B}(H)$ of bounded linear operators on $H$. Another important
example\footnote{%
\normalfont\sf
There are others, notable the ultrastrong and ultraweak operator topologies
\cite[Chapter 1]{KR}.} of a locally convex structure on $\mathcal{B}(H)$ has
defining seminorms $T\rightsquigarrow\left\vert \left\langle Tx,y\right\rangle
\right\vert $ indexed by the ordered pairs $\left(  x,y\right)  $ of vectors
in $H$ with $\left\Vert x\right\Vert ,\left\Vert y\right\Vert \leq1$; this
structure gives rise to the \textbf{weak operator topology }on $\mathcal{B}%
(H)$.

Returning to the general locally convex space $\left(  X,\left(  p_{i}\right)
_{i\in I}\right)  $, let $f$ be a mapping of a subset $S$ of $X$ into a
locally convex space $\left(  Y,\left(  q_{j}\right)  _{j\in J}\right)  $.
Then $f$ is (pointwise) continuous, in the usual topological sense, at $a\in
S\ $if for each finitely enumerable subset $G$ of $J$ and each $\varepsilon
>0$, there exist a finitely enumerable subset $F$ of $I$ and $\delta>0$ such
that if $x\in S$ and $\sum_{i\in F}p_{i}(x-a)<\delta$, then $\sum_{j\in
J}q_{j}(f(x)-f(a))<\varepsilon$. We say that $f$ is \textbf{uniformly
continuous} on $S$ if for each finitely enumerable subset $G$ of $J$ and each
$\varepsilon>0$, there exist a finitely enumerable subset $F$ of $I$ and
$\delta>0$ such that if $x,y\in S$ and $\sum_{i\in F}p_{i}(x-y)<\delta$, then
$\sum_{j\in J}q_{j}(f(x)-f(y))<\varepsilon$.

\begin{proposition}
\label{jan22p2}Let $S$ be a subset of the locally convex space $\left(
X,\left(  p_{i}\right)  _{i\in I}\right)  $,\ let $s\in X$, and let $F$ be a
finitely enumerable subset of $I$. Then the mapping $x\rightsquigarrow
\sum_{i\in F}p_{i}(x-s)$ of $X$ into $\mathbb{R}$ is uniformly continuous. In
particular, for each $i\in I$ the mapping $p_{i}$ is uniformly continuous on
$X$.
\end{proposition}

\begin{proof}
For each $\varepsilon>0$, if $x,y\in X$ and $\sum_{i\in F}p_{i}%
(x-y)<\varepsilon$, then (since the $p_{i}$ are seminorms)%
\begin{align*}
\left\vert \sum_{i\in F}p_{i}(x-s)-\sum_{i\in F}p_{i}(y-s)\right\vert  &
\leq\sum_{i\in F}\left\vert p_{i}(x-s)-p_{i}(y-s)\right\vert \\
&  \leq\sum_{i\in F}p_{i}((x-s)-(y-s))=\sum_{i\in F}p_{i}(x-y)<\varepsilon.
\end{align*}
The final conclusion of the proposition is just the case $s=0$.
\end{proof}

%

\medskip

In Proposition 5.4.1 of \cite{BVtech} and its proof, every instance of
\emph{linear mapping }should be replaced by \emph{linear functional. }Here is
the correct statement and proof of the more general proposition.

\begin{proposition}
\label{jan20p1}Let $\left(  X,\left(  p_{i}\right)  _{i\in I}\right)  $ and
$\left(  Y,\left(  q_{j}\right)  _{j\in J}\right)  $ be locally convex spaces.
Then the following are equivalent conditions on a linear mapping
$u:X\rightarrow Y$.

\begin{enumerate}
\item[\emph{(a)}] $u$ is continuous at $0$.

\item[\emph{(b)}] $u$ is continuous on $X$.

\item[\emph{(c)}] $u$ is uniformly continuous on $X$.

\item[\emph{(d)}] For each finitely enumerable subset $G$ of $J$ there exist
$C>0$ and a finitely enumerable set $F\subset I$ such that%
\begin{equation}
\sum_{j\in J}q_{j}(u(x))\leq C\sum_{i\in F}p_{i}(x)\ \ \ (x\in X). \label{30a}%
\end{equation}

\end{enumerate}
\end{proposition}

\begin{proof}
It is routine to show that (d)$\ \Rightarrow\ $(c)\ $\Rightarrow
\ $(b)\ $\Rightarrow\ $(a). To complete the proof, suppose that $u$ is
continuous at $0$, and let $G$ be a finitely enumerable subset of $J$. There
exist $C>0$ and a finitely enumerable set $F\subset I$ such that if
$\sum_{i\in F}p_{i}(x)\leq C^{-1}$, then $\sum_{j\in J}q_{j}(x)<1$. For each
$x\in X$ and each $\varepsilon>0$ we have%
\[
\sum_{i\in F}p_{i}\left(  \frac{C^{-1}x}{\varepsilon+\sum_{i\in F}p_{i}%
(x)}\right)  \leq C^{-1},
\]
so%
\[
\sum_{j\in J}q_{j}\left(  u\left(  \frac{C^{-1}x}{\sum_{i\in F}p_{i}%
(x)+\varepsilon}\right)  \right)  <1\text{ }%
\]
and therefore%
\[
\sum_{j\in J}q_{j}(u(x))<C\left(  \sum_{i\in F}p_{i}(x)+\varepsilon\right)  .
\]
Since $x\in X$ and $\varepsilon>0$ are arbitrary, (\ref{30a}) now follows.
\end{proof}

%

\medskip

Let $S$ be a subset of the locally convex space $X$. If $F$ is a finitely
enumerable subset of $I$ and%
\[
\rho^{F}(x,S)\equiv\inf\left\{  \sum_{i\in F}p_{i}(x-y):y\in S\right\}
\]
exists, then $F$ is $F$\textbf{-located} in $X$. If this holds for all
finitely enumerable $F\subset I$, then $S$ is \textbf{located }in $X$.

\begin{lemma}
\label{l32}Let $\left(  X,\left(  p_{i}\right)  _{i\in I}\right)  $ be a
locally convex space, and $u$ a nonzero continuous linear functional on $X$.
Let $F$ be a finitely enumerable subset of $I$, and $C>0$ be such that
$\left\vert u(x)\right\vert \leq C\sum_{i\in F}p_{i}(x)$ for each $x\in X$.
Then $\ker u$ is $F$-located if and only if%
\[
s_{x}\equiv\inf\left\{  t>0:u(x)\in tu(\overline{B}^{F}(1))\right\}
\]
exists for each $x\in X$. In that case, if $u(x)\neq0$, then $s_{x}>0$.
\end{lemma}

\begin{proof}
For each $x\in X$ and each $t>0$,%
\begin{align*}
\exists y\in\ker u\left(
{\textstyle\sum_{i\in F}}
p_{i}(x-y)\leq t\right)   &  \Leftrightarrow\exists z\in X\left(  \sum_{i\in
F}p_{i}(z)\leq1\text{ and }u(x-tz)=0\right) \\
&  \Leftrightarrow u(x)\in tu(\overline{B}^{F}(1)).
\end{align*}
Thus%
\[
\left\{  t>0:\exists y\in\ker u\left(
{\textstyle\sum_{i\in F}}
p_{i}(x-y)\leq t\right)  \right\}  =\left\{  t>0:u(x)\in tu(\overline{B}%
^{F}(1))\right\}  .
\]
Since $\rho^{F}(x,\ker u)$, if it exists, equals the infimum of the left-hand
set, and $s_{x}$, if it exists, equals the infimum of the right-hand set, the
first part of the lemma now follows.

Supposing that $s_{x}$ exists and that $u(x)\neq0$, let $0<r<C^{-1}\left\vert
u(x)\right\vert $. If $s_{x}<r$, then there exist a positive $t<r$ and
$z\in\overline{B}^{F}(1)$ such that $u(x)=tu(z)=u(tz)$. But $\sum_{i\in
F}p_{i}(tz)=t\sum_{i\in F}p_{i}(z)\leq t<r$, so $\left\vert u(tz)\right\vert
\leq C\sum_{i\in F}p_{i}(tz)<Cr<\left\vert u(x)\right\vert $, a contradiction.
Hence $s_{x}\geq r>0$.
\end{proof}

%

\medskip

Now let $u$ be a continuous linear functional on our locally convex space $X$,
and let $F\subset I$ be finitely enumerable. If%
\[
\left\Vert u\right\Vert _{F}\equiv\sup\left\{  \left\vert u(x)\right\vert
:x\in X,\,\sum_{i\in F}p_{i}(x)\leq1\right\}
\]
exists, we say that $u$ is $F$-\textbf{normed}, or $F$-\textbf{normable}, and
we call $\left\Vert u\right\Vert _{F}$ the $F$\textbf{-norm} of $u$. In that
case,
\begin{equation}
\left\Vert u\right\Vert _{F}=\sup\left\{  \left\vert u(x)\right\vert :x\in
X,\,\sum_{i\in F}p_{i}(x)=1\right\}  \label{zz1}%
\end{equation}
provided the set on the right of (\ref{zz1}) is inhabited. If this holds for
all finitely enumerable $F\subset I$, then $u$ is \textbf{normed}, or
\textbf{normable}.

Note that if $u$ is $F$-normed, then
\[
\left\vert u(x)\right\vert \leq\left\Vert u\right\Vert _{F}\sum_{i\in F}%
p_{i}(x)\ \ \ (x\in X).
\]
Moreover, if $C>0$ and $\left\vert u(x)\right\vert \leq C\sum_{i\in F}%
p_{i}(x)$ for all $x\in X$, then $c\geq\left\Vert u\right\Vert _{F}$. For if
$C<\left\Vert u\right\Vert _{F}$, then there exists $x\in X$ such that
$\sum_{i\in F}p_{i}(x)\leq1$ and $C<\left\vert u(x)\right\vert \leq C$, which
is absurd. Hence%
\[
\left\Vert u\right\Vert _{F}=\inf\left\{  C>0:\left\vert u(x)\right\vert \leq
C\sum_{i\in F}p_{i}(x)\ \ (x\in X)\right\}  .
\]

The following is the locally convex space analogue of the standard criterion
for normability for linear functionals on normed spaces in \cite[Proposition
8, page 258]{Bishop},\cite[2.3.6]{BVtech}.

\begin{proposition}
\label{p33}Let $\left(  X,\left(  p_{i}\right)  _{i\in I}\right)  $ be a
locally convex space, and $u$ a nonzero continuous linear functional on $X$.
Let $F$ be a finitely enumerable subset of $I$, and $C>0$ such that
$\left\vert u(x)\right\vert \leq C\sum_{i\in F}p_{i}(x)$ for each $x\in X$.
Then $u$ is $F$-normed if and only if $\ker u$ is $F$-located in $X$.
\end{proposition}

\begin{proof}
Suppose first that $u$ is $F$-normed; then since $u$ is nonzero, $\left\Vert
u\right\Vert _{F}>0$. Consider any $a\in X$. For each $y\in\ker u$ we have%
\[
\sum_{i\in F}p_{i}(a-y)\geq\frac{\left\vert u(a-y)\right\vert }{\left\Vert
u\right\Vert _{F}}=\frac{\left\vert u(a)\right\vert }{\left\Vert u\right\Vert
_{F}}.
\]
On the other hand, if $0<\varepsilon<\left\Vert u\right\Vert _{F}$ and we
choose $x\in X$ such that $\sum_{i\in F}p_{i}(x)=1$ and $u(x)>\left\Vert
u\right\Vert _{F}-\varepsilon$, then%
\[
z\equiv a-\frac{u(a)}{u(x)}x\in\ker u
\]
and%
\[
\sum_{i\in F}p_{i}(a-z)=\sum_{i\in F}p_{i}\left(  \frac{u(a)}{u(x)}x\right)
=\frac{u(a)}{u(x)}\sum_{i\in F}p_{i}\left(  x\right)  =\frac{u(a)}{u(x)}%
<\frac{\left\vert u(a)\right\vert }{\left\Vert u\right\Vert _{F}-\varepsilon
}.
\]
Since $\varepsilon$ is arbitrary, it follows that $\inf_{y\in\ker u}\sum_{i\in
F}p_{i}(a-y)$ exists and equals $\left\vert u(a)\right\vert /\left\Vert
u\right\Vert _{F}$. Since $a$ is arbitrary, we see that $\ker u$ is $F$-located.

Conversely, suppose that $\ker u$ is $F$-located. Since $u$ is nonzero, there
exists $x_{0}$ with $u(x_{0})=1$. Thus, by Lemma \ref{l32},%
\[
s\equiv\inf\left\{  t>0:1\in tu(\overline{B}^{F}(1))\right\}
\]
exists and is positive. We show that $\left\Vert u\right\Vert _{F}$ equals
$1/s$. For each $x\in\overline{B}^{F}(1)$ we have either $u(x)<1/s$ or
$u(x)\neq0$. In the latter case,
\[
\sum_{i\in F}p_{i}\left(  \frac{\left\vert u(x)\right\vert }{u(x)}x\right)
\leq1
\]
and%
\[
1=\frac{1}{\left\vert u(x)\right\vert }u\left(  \frac{\left\vert
u(x)\right\vert }{u(x)}x\right)  ,
\]
so $1/\left\vert u(x)\right\vert \geq s$ and therefore $\left\vert
u(x)\right\vert \leq1/s$. On the other hand, by definition of $s$, if
$0<\varepsilon<1/s$, then there exist $t$ with $s\leq t<s/\left(
1-\varepsilon s\right)  $, and $z\in\overline{B}^{F}(1)$, such that $tu(z)=1$.
Then $u(z)=1/t>1/s>1/s-\varepsilon$. Since $\varepsilon$ is arbitrary, it
follows that $\left\Vert u\right\Vert _{F}$ exists and equals $1/s$.
\end{proof}

\section{Total boundedness in locally convex spaces}

Let $S$ be an inhabited subset of our locally convex space\footnote{%
\normalfont\sf
The work in this section up to the end of the proof of Corollary \ref{c22}~can
readily be adapted to apply not just to locally convex spaces, but to sets
with a uniform structure defined by a family of pseudonorms. Such spaces are
introduced in Problems 17--20 on pages 110-111 of \cite{Bishop}.} $X$, let
$F\subset I$ be finitely enumerable, and let $\varepsilon>0$. By an $\left(
F,\varepsilon\right)  $\textbf{-approximation} to $S$ we mean a set $T\subset
S$ such that for each $x\in S$ there exists $y\in T$ with $\sum_{i\in F}%
p_{i}(x-y)<\varepsilon$. We say that $S$ is $F$\textbf{-totally bounded} if
for each $\varepsilon>0$ there exists a finitely enumerable $\left(
F,\varepsilon\right)  $-approximation to $S$. If $S$ is $F $-totally bounded
for each finitely enumerable set $F\subset I$, then we say that $S$ is
\textbf{totally bounded }in $X$. If $\left(  X,\left\Vert \ \right\Vert
\right)  $ is a normed linear space, then the total boundedness of a subset of
the locally convex space $X$ is just total boundedness in the usual sense
relative to the metric induced on $X$ by $\left\Vert \ \right\Vert $.

\begin{proposition}
\label{feb02p1}Let $\left(  X,\left(  p_{i}\right)  _{i\in I}\right)  $ be a
locally convex space, $F$ a finitely enumerable subset of $I$, and $S$ an
$F$-totally bounded subset of $X$. Then $S$ is $F$-\textbf{bounded}, in the
sense that there exists $c>0$ such that $\sum_{i\in F}p_{i}(x)\leq c$ for all
$x\in S$.
\end{proposition}

\begin{proof}
Choose a finitely enumerable $\left(  F,1\right)  $-approximation $\left\{
x_{1},\ldots,x_{N}\right\}  $ to $S$, and let%
\[
c=1+\sum_{i\in F}\sum_{j=1}^{N}p_{i}(x_{j}).
\]
Given $x\in S$, choose $k$ such that $\sum_{i\in F}p_{i}(x-x_{k})<1$. Then
\begin{align*}
\sum_{i\in F}p_{i}(x) &  \leq\sum_{i\in F}(p_{i}(x-x_{k})+p_{i}(x_{k}))\\
&  \leq\sum_{i\in F}p_{i}(x-x_{k})+\sum_{i\in F}p_{i}(x_{k})\leq1+\sum_{i\in
F}\sum_{j=1}^{N}p_{i}(x_{j})=c
\end{align*}
%

\hfill

\end{proof}

\begin{proposition}
\label{jan20p2}If $S$ is a totally bounded subset of the locally convex space
$X$, and $f$ is a uniformly continuous mapping of $S$ into the locally convex
space space $\left(  Y,\left(  q_{j}\right)  _{j\in J}\right)  $, then $f(X)$
is totally bounded in $Y$.
\end{proposition}

\begin{proof}
Let $G\subset J$ be finitely enumerable and $\varepsilon>0$. There exist a
finitely enumerable subset $F$ of $I$ and $\delta>0$ such that if $x,y\in S$
and $\sum_{i\in F}p_{i}(x-y)<\delta$, then $\sum_{j\in J}q_{j}%
(f(x)-f(y))<\varepsilon$. Let $T$ be a finitely enumerable $(F,\delta
)$-approximation to $S$. Then for each $x\in S$ there exists $x^{\prime}\in T$
such that $\sum_{i\in F}p_{i}(x-x^{\prime})<\delta$ and therefore $\sum_{j\in
J}q_{j}(f(x)-f(x^{\prime}))<\varepsilon$. Thus $f(T)$ is a finitely enumerable
$\left(  G,\varepsilon\right)  $-approximation to $f(S)$.
\end{proof}

\begin{corollary}
\label{jan22c1}Let $u$ be a continuous linear mapping of the locally convex
space $\left(  X,\left(  p_{i}\right)  _{i\in I}\right)  $ into the locally
convex space space $\left(  Y,\left(  q_{j}\right)  _{j\in J}\right)  $. If
$S\subset X$ is totally bounded, then $u(S)$ is totally bounded in $Y$.
\end{corollary}

\begin{proof}
By Proposition \ref{jan20p1}, $u$ is uniformly continuous on $X$, so the
restriction of $u$ to $S$ is uniformly continuous. It remains to apply
Proposition \ref{jan20p2}.
\end{proof}

\begin{corollary}
\label{jan20c2}If $S$ is a totally bounded subset of the locally convex space
$X$, and $f$ is a uniformly continuous mapping of $S$ into $\mathbb{R}$, then
$\sup_{x\in S}f(x)$ and $\inf_{x\in S}f(x)$ exist.
\end{corollary}

\begin{proof}
By Proposition \ref{jan20p2} and the remark immediately preceding that
proposition, $f(S)$ is totally bounded in $\mathbb{R}$ in the usual metric
sense. Hence \cite[Proposition 2.2.5]{BVtech} can be applied.
\end{proof}

\begin{proposition}
\label{jan22p1}A totally bounded subset $K$ of the locally convex space $X$ is located.
\end{proposition}

\begin{proof}
It follows from Proposition \ref{jan22p2} and Corollary \ref{jan20c2} that if
$x\in X$ and $F\subset I$ is finitely enumerable, then $\rho^{F}(x,K)$ exists.
\end{proof}

\begin{proposition}
\label{jan23p1}Let $K$ be a totally bounded subset of the locally convex space
$X$, and let $S\subset K$ be located in $K$. Then $S$ is totally bounded.
\end{proposition}

\begin{proof}
Let $F\subset I$ be finitely enumerable and $\varepsilon>0$. Construct a
finitely enumerable $\left(  F,\varepsilon/3\right)  $-approximation $\left\{
x_{1},\ldots,x_{n}\right\}  $ to $K$, and for each $k\leq n$ let $\rho
_{k}=\rho^{F}(x_{k},S)$. Write $\left\{  1,\ldots,n\right\}  $ as a union of
two sets $P$ and $Q$ where $\rho_{k}<2\varepsilon/3$ if $k\in P$, and
$\rho_{k}>\varepsilon/3$ if $k\in Q$. For each $k\in P$ there exists $s_{k}\in
S$ such that $\sum_{i\in F}p_{i}(x_{k}-s_{k})<2\varepsilon/3$. Given $s\in S$,
choose $k$ such that $\sum_{i\in F}p_{i}(s-x_{k})<\varepsilon/3$. Then
$\rho_{k}<\varepsilon/3$, so $k\notin Q$ and therefore $k\in P$; whence
\begin{align*}
\sum_{i\in F}p_{i}(s-s_{k}) &  \leq\sum_{i\in F}\left(  p_{i}(s-x_{k}%
)+p_{i}(x_{k}-s_{k})\right)  \\
&  \leq\sum_{i\in F}p_{i}(s-x_{k})+\sum_{i\in F}p_{i}(x_{k}-s_{k}%
)<\frac{\varepsilon}{3}+\frac{2\varepsilon}{3}=\varepsilon.
\end{align*}
Thus $\left\{  s_{k}:k\in P\right\}  $ is a finitely enumerable $\left(
F,\varepsilon\right)  $-approximation to $S$. Since $F$ and $\varepsilon$ are
arbitrary, it follows that $S$ is totally bounded.
\end{proof}

%

\bigskip

If $K\subset X$ is $F$-totally bounded, we define%
\[
\mathsf{diam}^{F}(K)\equiv\sup\left\{  \sum_{i\in F}p_{i}(x-y):x,y\in
K\right\}  ,
\]
which exists by Proposition \ref{jan22p2} and Corollary \ref{jan20c2}.

Observe here that if $x,y\in X$, then for each $z\in K$,%
\[
\rho^{F}(x,K)\leq\sum_{i\in F}p_{i}(x-z)\leq\sum_{i\in F}p_{i}(x-y)+\sum_{i\in
F}p_{i}(y-z),
\]
and therefore%
\begin{align*}
\rho^{F}(x,K)  &  \leq\sum_{i\in F}p_{i}(x-y)+\inf\left\{  \sum_{i\in F}%
p_{i}(y-z):z\in K\right\} \\
&  =\sum_{i\in F}p_{i}(x-y)+\rho^{F}(y,K).
\end{align*}

The next three results and their proofs are locally-convex-space-analogues of
ones for metric spaces (see \cite[pages 42--44]{BVtech}).

\begin{proposition}
\label{p11}Let $S$ be a totally bounded subset of the locally convex space
$\left(  X,\left(  p_{i}\right)  _{i\in I}\right)  $, and $x_{0}\in S$. Let
$F$ be a finitely enumerable subset of $I$, and $r>0$. Then there exists a
closed $F$-totally bounded subset $K$ of $S$ such that $S\cap B^{F}%
(x_{0},r)\subset K\subset S\cap\overline{B}^{F}(x_{0},8r)$.
\end{proposition}

\begin{proof}
With $S_{1}\equiv\{x_{0}\},$ we construct inductively a sequence
$(S_{n})_{n\geqslant1}$ of finitely enumerable subsets of $S$ such that

\begin{itemize}
\item[(a)] $\rho^{F}(x,S_{n})<2^{-n+1}r$ for each $x$ in $S\cap B^{F}%
(x_{0},r)$, and

\item[(b)] \smallskip$\rho^{F}(x,S_{n})<2^{-n+3}r$ for each $x$ in $S_{n+1}.$
\end{itemize}

%

\noindent
To that end, assume that $S_{1},\ldots,S_{n}$ have been constructed with the
appropriate properties, and let $\{x_{1},\ldots,x_{N}\}$ be an $(F,2^{-n}%
r)$-approximation to $S.$ Write $\{1,\ldots,N\}$ as a union of subsets $A$ and
$B$ such that%
\begin{align*}
\rho^{F}(x_{k},S_{n})  &  <2^{-n+3}r\quad\mathrm{if}\,\,k\in A,\\
\rho^{F}(x_{k},S_{n})  &  >2^{-n+2}r\quad\mathrm{if}\,\,k\in B.
\end{align*}
Let%
\[
S_{n+1}\equiv\{x_{k}:k\in A\}.
\]
Clearly, $S_{n+1}$ satisfies (b). Let $x$ be any point of $S\cap B^{F}%
(x_{0},r).$ By the induction hypothesis, there exists $y$ in $S_{n}$ with
$\sum_{i\in F}p_{i}(x-y)<2^{-n+1}r.$ Choosing $k$ in $\{1,\ldots,N\}$ such
that $\sum_{i\in F}p_{i}(x-x_{k})<2^{-n}r,$ we have
\begin{align*}
\rho^{F}(x_{k},S_{n})  &  \leqslant\sum_{i\in F}p_{i}(x_{k}-y)\\
&  \leqslant\sum_{i\in F}p_{i}(x-x_{k})+\sum_{i\in F}p_{i}(x-y)\\
&  <2^{-n}r+2^{-n+1}r\\
&  <2^{-n+2}r.
\end{align*}
Thus $k\notin B$, so $k\in A$, which is therefore (inhabited and) finitely
enumerable. Moreover,
\[
\rho^{F}(x,S_{n+1})\leq\sum_{i\in F}p_{i}(x-x_{k})<2^{-(n+1)+1}r,
\]
so $S_{n+1}$ satisfies the appropriate instance of (a). This completes the
inductive construction. Letting $K$ be the closure of $\bigcup\limits_{n=1}%
^{\infty}S_{n}$ in $X$, we see from (a) that $S\cap B^{F}(x_{0},r)\subset K.$

Next we prove:

\begin{itemize}
\item[(*)] \emph{If }$m,n$\emph{\ are positive integers with }$m\geq n$\emph{,
and }$y\in S_{m}$\emph{, then there exists }$y_{n}\in S_{n}$\emph{\ such that
}$\sum_{i\in F}p_{i}(y-y_{n})<2^{-n+4}r$\emph{.}
\end{itemize}

%

\noindent
Given $y\in S_{m}$ and letting $y_{m}=y$, we see from (b) that for each $k$
with $n\leq k\leq m-1$ there exist points $y_{k}\in S_{k}$ such that
$\sum_{i\in F}p_{i}(y_{k+1}-y_{k})<2^{-k+3}r$. Then%
\[
\sum_{i\in F}p_{i}(y-y_{n})\leqslant\sum_{k=n}^{m-1}\sum_{i\in F}p_{i}%
(y_{k+1}-y_{k})<\sum_{k=n}^{\infty}2^{-k+3}r=2^{-n+4}r.
\]
In particular taking $n=1$, we see that $\rho^{F}(y,\left\{  x_{0}\right\}
)=\sum_{i\in F}p_{i}(y-y_{1})<2^{3}r~$for each $y\in\bigcup\limits_{i=1}%
^{\infty}S_{i}$; whence $\bigcup\limits_{n=1}^{\infty}S_{n}\subset S\cap
B^{F}(x_{0},8r)$ and therefore $K\subset\overline{B}^{F}(x_{0},8r)$. On the
other hand, given $x\in K$ and a positive integer $n$, we can find $m$ and
$y\in S_{m}$ such that $\sum_{i\in F}p_{i}(x-y)<2^{-n+4}r$. If $m<n$, then
$y\in\bigcup\limits_{k=1}^{n}S_{k\text{.}}$. If $m\geq n$, then by the remark
preceding this proposition,%
\[
\rho^{F}(x,S_{n})\leqslant\sum_{i\in F}p_{i}(x-y)+\rho^{F}(y,S_{n}%
)<2^{-n+4}r+2^{-n+4}r=2^{-n+5}r,
\]
the second inequality using (*). Thus there exists $z\in S_{n}\subset
\bigcup\limits_{k=1}^{n}S_{k\text{. }}$ such that $\sum_{i\in F}%
p_{i}(x-z)<2^{-n+5}r$. Putting together the two alternatives for $m$, we now
see that the finitely enumerable set $\bigcup\limits_{k=1}^{n}S_{k\text{. }}%
$is a $\left(  F,2^{-n+5}r\right)  $-approximation to $K$. Since $n$ is
arbitrary, we conclude that $K$ is $F$-totally bounded.
\end{proof}

\begin{corollary}
\label{c11}Let $S$ be a totally bounded subset of the locally convex space
$\left(  X,\left(  p_{i}\right)  _{i\in I}\right)  $, and $F$ a finitely
enumerable subset of $I$. Then for each $\varepsilon>0$ there exist finitely
many $F$-totally bounded sets $K_{1},\ldots,K_{N}$ such that $S=%
{\textstyle\bigcup_{n=1}^{N}}
K_{n}$ and $\mathsf{diam}^{F}(K_{n})\leq\varepsilon$ for each $n\leq N$.
\end{corollary}

\begin{proof}
Given\textbf{\emph{\ }}$\varepsilon>0$, construct an $\left(  F,\varepsilon
/16\right)  $-approximation $\left\{  x_{1},\ldots,x_{N}\right\}  $ to $S$. By
Proposition \ref{p11}, for each $n\in\left\{  1,\ldots,N\right\}  $ there
exists a closed, $F$-totally bounded set $K_{n}$ such that $S\cap B^{F}%
(x_{n},\varepsilon/16)\subset K_{n}\subset S\cap\overline{B}^{F}%
(x_{n},\varepsilon/2)$. Hence
\[
S\subset\bigcup_{n=1}^{N}\left(  S\cap B^{F}\left(  x_{n},\tfrac{\varepsilon
}{16}\right)  \right)  \subset%
{\textstyle\bigcup_{n=1}^{N}}
K_{n}.
\]
Also, for all $x,y\in K_{n}$ we have%
\[
\sum_{n=1}^{N}p_{n}(x-y)\leq\sum_{n=1}^{N}p_{n}(x-x_{n})+\sum_{n=1}^{N}%
p_{n}(y-x_{n})\leq\tfrac{\varepsilon}{2}+\tfrac{\varepsilon}{2}=\varepsilon
\text{,}%
\]
so $\mathsf{diam}^{F}(K_{n})\leq\varepsilon$.
\end{proof}

\begin{theorem}
\label{t11}Let $S$ be a totally bounded subset of the locally convex space
$\left(  X,\left(  p_{i}\right)  _{i\in I}\right)  $, $f$ a uniformly
continuous mapping of $S$ into $\mathbb{R}$, and $F$ a finitely enumerable
subset of $I$. Then for all but countably many $r\in\mathbb{R}$ the set$.$%
\[
S(f,r)\equiv\left\{  x\in S:f(x)\leq r\right\}
\]
is either $F$-totally bounded or empty.
\end{theorem}

\begin{proof}
By Corollary \ref{c11}, for each positive integer $k$ there exist a positive
integer $n_{k}$ and $F$-totally bounded sets $S_{kj}\ \left(  1\leq j\leq
n_{k}\right)  $, with each $\mathsf{diam}^{F}(S_{kj})<1/k$, whose union is
$S$. Let $\left(  r_{n}\right)  _{n\geq1}$ be an enumeration of the real
numbers%
\[
c_{kj}\equiv\inf\left\{  f(x):x\in S_{kj}\right\}  \ \ (1\leq k;\,1\leq j\leq
n_{k}),
\]
each of which exists by Corollary \ref{jan20c2}. Consider any $r\in\mathbb{R}$
such that $r\neq r_{n}$ for each $n$. For each positive integer $k$, since
$r\ $is distinct from each $c_{kj}$, either $c_{kj}>r$ for each $j\leq n_{k}$,
in which case $S(f,r)=%
{\textstyle\bigcup_{j=1}^{n_{k}}}
S_{kj}(f,r)=\varnothing$, or else%
\[
S_{k}\equiv\left\{  j:1\leq j\leq n_{k},\,c_{kj}<r\right\}
\]
is inhabited and therefore finitely enumerable. In the latter case, for each
$j\in S_{k}$ choose $x_{kj}\in S_{kj}$. Given $x\in S(f,r)$, choose $j\leq
n_{k}$ such that $x\in S_{kj}$. Then $c_{kj}\leq f(x)\leq r$, so $c_{kj}<r$
(since $r\neq c_{kj}$) and therefore $j\in S_{k}$. Hence%
\[
\sum_{i\in F}p_{i}(x-x_{kj})\leq\mathsf{diam}^{F}(S_{kj})<\frac{1}{k}.
\]
Thus%
\[
\left\{  x_{kj}:1\leq j\leq n_{k}\text{, }c_{kj}<r\right\}
\]
is a finitely enumerable $(F,1/k)$-approximation to $S$. Since $k$ is
arbitrary, $S$ is $F$-totally bounded.
\end{proof}

\begin{corollary}
\label{c22}Let $\left(  X,\left(  p_{n}\right)  _{n\geq1}\right)  $ be a
locally convex space with a countable family $\left(  p_{n}\right)  _{n\geq1}$
of defining seminorms, $S$ a totally bounded subset of $X$, and $f$ a
uniformly continuous mapping of $S$ into $\mathbb{R}$. Then for all but
countably many $r\in\mathbb{R}$ the set $S(f,r)$ is either totally bounded or empty.
\end{corollary}

\begin{proof}
By Theorem \ref{t11}, for each positive integer $N$ there exists a countable
family $\left(  c_{N,k}\right)  _{k\geq1}$ of real numbers such that if
$r\in\mathbb{R}$ and $r\neq c_{N,k}$ for each $k$, then $S(f,r)$ is $\left\{
1,\ldots,N\right\}  $-totally bounded. Let $\left(  c_{k}\right)  _{k\geq1}$
be an enumeration of all the numbers $c_{N,k}\ (N,k\geq1)$, and let
$r\in\mathbb{R}$ satisfy $r\neq c_{k}$ for each $k\geq1$. If $F$ is any
finitely enumerable set of positive integers, then there exists $N$ such that
$F\subset\left\{  1,\ldots,N\right\}  $. Given $\varepsilon>0$, choose a
finitely enumerable $(\left\{  1,\ldots,N\right\}  ,\varepsilon)$%
-approximation $Y$ to $S$. For each $x\in S$ there exists $y\in Y$ such that
$\sum_{n\in F}p_{i}(x-y)\leq\sum_{n=1}^{N}p_{i}(x-y)<\varepsilon$. Hence $Y$
is an $(F,\varepsilon)$-approximation to $S$. Since $F$ and $\varepsilon\ $are
arbitrary, it follows that $S$ is totally bounded.
\end{proof}

%

\medskip

The key application of these results is the following.\footnote{%
\normalfont\sf
This proposition improves the statement and proof of \cite[5.4.9]{BVtech}.}

\begin{proposition}
\label{p22}Let $\left(  X,\left(  p_{i}\right)  _{i\in I}\right)  $ be a
locally convex space; $K$ a balanced, convex, totally bounded subset of $X$;
and $u$ a nonzero continuous linear functional on $X$. Then $K\cap\ker u$ is
totally bounded.
\end{proposition}

\begin{proof}
Since, by Proposition \ref{jan20p1}, $u$ is uniformly continuous on the
totally bounded set $K$, we see that%
\[
C\equiv\sup\left\{  \left\vert u(x)\right\vert :x\in K\right\}
\]
exists by Corollary \ref{jan20c2}. Clearly, $C$ is positive. Choose $y_{1}\in
K$ such that $u(y_{1})>C/2$. Then%
\[
y_{0}\equiv\frac{C}{2u(y_{1})}y_{1}%
\]
belongs to the balanced set $K$, and $u(y_{0})=C/2$. Let $\varepsilon>0$ and
let $F$ be a finitely enumerable subset of $I$. By Propositions \ref{jan22p2}
and \ref{jan20p2}, each $p_{i}(K)$ is a totally bounded subset of $\mathbb{R}%
$, so by Proposition \ref{feb02p1}, there exists $b>0$ such that $\sum_{i\in
F}p_{i}(x)\leq b$ for each $x\in K$. By Theorem \ref{t11}, there exists $t$
such that%
\[
0<t<\frac{C\varepsilon}{C+4b}%
\]
and the set%
\[
S_{t}\equiv\left\{  y\in K:\left\vert u(y)\right\vert \leq t\right\}
\]
is $F$-totally bounded. Pick an $\left(  F,t\right)  $-approximation $\left\{
s_{1},\ldots,s_{n}\right\}  $ to $S_{t}$, and set%
\[
y_{k}\equiv\frac{C}{C+2t}s_{k}-\frac{2}{C+2t}u(s_{k})y_{0}\ \ \ \left(  1\leq
k\leq n\right)  .
\]
Then $y_{k}\in\ker u$. Since $\left\vert u(s_{k})\right\vert \leq t$ and $K$
is balanced,%
\[
\frac{-u(s_{k})}{t}y_{0}\in K.
\]
Thus%
\[
y_{k}=\frac{C}{C+2t}s_{k}+\left(  1-\frac{C}{C+2t}\right)  \frac{-u(s_{k})}%
{t}y_{0}\in K
\]
and%
\[
s_{k}-y_{k}=\left(  1-\frac{C}{C+2t}\right)  \left(  s_{k}+\frac{u(s_{k})}%
{t}y_{0}\right)  =\frac{2t}{C+2t}\left(  s_{k}+\frac{u(s_{k})}{t}y_{0}\right)
\]
Also,%
\begin{align*}
\sum_{i\in F}p_{i}(s_{k}-y_{k}) &  =\frac{2t}{C+2t}\sum_{i\in F}p_{i}\left(
s_{k}+\frac{u(s_{k})}{t}y_{0}\right)  \\
&  \leq\frac{2t}{C}\sum_{i\in F}\left(  p_{i}(s_{k})+\frac{\left\vert
u(s_{k})\right\vert }{t}p_{i}(y_{0})\right)  \\
&  \leq\frac{2t}{C}\left(  \sum_{i\in F}p_{i}(s_{k})+\sum_{i\in F}p_{i}%
(y_{0})\right)  \leq\frac{2t}{C}(b+b)=\frac{4tb}{C}.
\end{align*}
If $y\in K\cap\ker u\subset S_{t}$, then there exists $k$ such that
$\sum_{i\in F}p_{i}(y-s_{k})<t$ and therefore%
\[
\sum_{i\in F}p_{i}(y-y_{k})\leq\sum_{i\in F}p_{i}(y-s_{k})+\sum_{i\in F}%
p_{i}(s_{k}-y_{k})<t+\frac{4tb}{C}=t\frac{C+4b}{C}<\varepsilon.
\]
Thus $\left\{  y_{1},\ldots,y_{n}\right\}  $ is a finitely enumerable $\left(
F,\varepsilon\right)  $-approximation to $K\cap\ker u$. Since $F,\varepsilon$
are arbitrary, it follows that $K$ is totally bounded.
\end{proof}

%

\medskip

\section{Applications}

We conclude by dealing, but without proofs, with some theorems that depend on
the work in our preceding sections. But first we have more definitions.

Let $\left(  X,\left(  p_{i}\right)  _{i\in I}\right)  $ be a \emph{separable}
locally convex space. We say that a sequence $\left(  x_{n}\right)  _{n\geq1}$
in $X$

\begin{itemize}
\item is a \textbf{Cauchy sequence} if for each $\varepsilon>0$ and each
finitely enumerable $F\subset I$ there exists $N$ such that $\sum_{i\in
F}p_{i}(x_{m}-x_{n})<\varepsilon$;

\item \textbf{converges} to the (perforce unique) \textbf{limit} $x_{\infty
}\in X$ if for each $\varepsilon>0$ and each finitely enumerable $F\subset I
$, there exists $N$ such that $\sum_{i\in F}p_{i}(x_{n}-x_{\infty
})<\varepsilon$.
\end{itemize}

%

\noindent
We say that the separable locally convex space $X$ is \textbf{complete} if
every Cauchy sequence in $X$ converges to a limit in $X$.%

\medskip
If $X$ and $Y$ are normed linear spaces, then the (norm-) \textbf{unit ball
}of $\mathcal{B}(X,Y)$ is the set%
\[
\mathcal{B}_{1}(X,Y)\equiv\left\{  T\in\mathcal{B}(X,Y):\left\Vert
Tx\right\Vert \leq\left\Vert x\right\Vert \text{ for all }x\in X\right\}  .
\]
The unit ball of the dual $X^{\ast}$ of $X$ is usually denoted by $X_{1}%
^{\ast}$, and that of $\mathcal{B}(H)$, where $H$ is a Hilbert space, by
$\mathcal{B}_{1}(H)$.

\begin{theorem}
\label{jan30t1}If $X$ is a separable normed linear space, then the
\textbf{unit ball,}%
\[
X_{1}^{\ast}\equiv\left\{  u\in X^{\ast}:\left\vert u(x)\right\vert
\leq\left\Vert x\right\Vert \text{ for all }x\in X\right\}  ,
\]
of the dual space $X_{1}^{\ast}\ $is complete and totally bounded relative to
the weak$^{\ast}$ topology (the \textbf{Banach-Alaoglu theorem}
\emph{\cite[5.4.7]{BVtech}).}
\end{theorem}

\begin{theorem}
\label{jan30t2}Let $X$ be a separable Banach space, and $f$ a weak$^{\ast}%
~$continuous linear functional on $X^{\ast}$. Then there exists $x\in X$ such
that $f(u)=u(x)$ for each $u\in X^{\ast}$ \emph{\cite[5.4.14]{BVtech}}
\end{theorem}

The proof of Theorem \ref{jan30t2} depends on Proposition \ref{p22} above, as
well as a number of technical lemmas on pages 134--137 of \cite{BVtech}%
.\footnote{%
\normalfont\sf
The statement of 5.4.14 on page 137 of \cite{BVtech} includes the hypothesis
that the linear functional is weak$^{\ast}$-uniformly continuous on
$X_{1}^{\ast}$, which, in view of our Proposition \ref{jan20p1}, is equivalent
to our hypothesis that it be weak$^{\ast}$ continuous.}

Another place where the results in Sections 1 and 2 are used is in the proof
of \cite[Theorem 10]{dsbfnl}:

\begin{theorem}
\label{jan30t3}Let $H$ be a nontrivial Hilbert space, and $u$ a nonzero
weak-operator continuous linear functional on $\mathcal{B}(H)$. Let $\delta$
be a positive number, $\xi_{1},\ldots,\xi_{N}$ linearly independent vectors in
$H$, and $\zeta_{1},\ldots,\zeta_{N}$ nonzero vectors in $H$, such that
$\left\vert u(T)\right\vert \leq\delta\sum_{n=1}^{N}\left\vert \left\langle
T\xi_{n},\zeta_{n}\right\rangle \right\vert $ for all $T\in\mathcal{B}(H)$.
Then there exist vectors $x_{k}\in\mathbb{C}\xi_{k}\ (1\leq k\leq N)$ such
that%
\[
u(T)=\sum_{n=1}^{N}\left\langle Tx_{n},\zeta_{n}\right\rangle
\]
for all $T\in\mathcal{B}(H)$.
\end{theorem}

\section{Concluding remarks}

The theory could go from here in at least two directions: firstly, developing
the more general theory of what Bishop calls \emph{uniform spaces }(see
\cite[pages 110-111]{Bishop} or \cite[page 124]{BB}), of which theory the
foregoing is a subset; secondly, investigating Bishop's uniform spaces as
apartness (uniform) spaces, as discussed in \cite{BVtech}. Be that as it may,
what we have developed in Sections 1--3 above is enough of the constructive
theory of locally convex spaces to enable results such as those in Section 4
(see also \cite{BridgesVNA}).%

\medskip

%

\bigskip
%

\noindent
\textbf{Author's address: \ }Department of Mathematics \& Statistics,
University of Canterbury, Christchurch 8140, New Zealand%

\noindent
\textbf{Author's email: \ \ \ \ \ }\texttt{dugbridges@gmail.com}

\end{document}